\documentclass[10pt,a4paper]{amsart}
\usepackage[margin=20mm]{geometry}
\usepackage[numbers]{natbib}
\usepackage{hyperref}

\usepackage{amssymb,amsmath}
\usepackage{amsthm}
\usepackage{bbm}
\usepackage{stmaryrd}
\usepackage[usenames,dvipsnames]{xcolor}
\usepackage{color}

\input xy
\xyoption{all} 

\numberwithin{equation}{section}

\newtheorem{theorem}{Theorem}[section]

\newtheorem{proposition}[theorem]{Proposition}

\theoremstyle{definition}
\newtheorem{definition}[theorem]{Definition}
\newtheorem{remark}[theorem]{Remark}
\newtheorem{example}[theorem]{Example}

\newcommand{\Dleft}{[\hspace{-1.5pt}[}
\newcommand{\Dright}{]\hspace{-1.5pt}]}
\newcommand{\SN}[1]{\Dleft #1 \Dright}

\newcommand{\p}{\mbox{\boldmath$\rho$}}
\newcommand{\Bs}{\mbox{\boldmath$s$}}

\newcommand{\rmi}{ \textnormal{i}}
\newcommand{\rmd}{\textnormal{d}}
\newcommand{\rme}{\textnormal{e}}
\newcommand{\rmh}{\textnormal{h}}

\DeclareMathOperator{\Vect}{Vect}

\DeclareMathOperator{\Ber}{Ber}

\DeclareMathOperator{\str}{str}
\DeclareMathOperator{\tr}{tr}

\font\black=cmbx10 \font\sblack=cmbx7 \font\ssblack=cmbx5 \font\blackital=cmmib10  \skewchar\blackital='177
\font\sblackital=cmmib7 \skewchar\sblackital='177 \font\ssblackital=cmmib5 \skewchar\ssblackital='177
\font\sanss=cmss10 \font\ssanss=cmss8 
\font\sssanss=cmss8 scaled 600 \font\blackboard=msbm10 \font\sblackboard=msbm7 \font\ssblackboard=msbm5
\font\caligr=eusm10 \font\scaligr=eusm7 \font\sscaligr=eusm5  \font\fraktur=eufm10
\font\sfraktur=eufm7 \font\ssfraktur=eufm5 
\font\bsymb=cmsy10 scaled\magstep2
\def\all#1{\setbox0=\hbox{\lower1.5pt\hbox{\bsymb
       \char"38}}\setbox1=\hbox{$_{#1}$} \box0\lower2pt\box1\;}
\def\exi#1{\setbox0=\hbox{\lower1.5pt\hbox{\bsymb \char"39}}
       \setbox1=\hbox{$_{#1}$} \box0\lower2pt\box1\;}

\def\tx#1{{\fam0\relax#1}}

\newfam\bifam
\textfont\bifam=\blackital \scriptfont\bifam=\sblackital \scriptscriptfont\bifam=\ssblackital

\newfam\blfam
\textfont\blfam=\black \scriptfont\blfam=\sblack \scriptscriptfont\blfam=\ssblack

\newfam\bbfam
\textfont\bbfam=\blackboard \scriptfont\bbfam=\sblackboard \scriptscriptfont\bbfam=\ssblackboard

\newfam\ssfam
\textfont\ssfam=\sanss \scriptfont\ssfam=\ssanss \scriptscriptfont\ssfam=\sssanss
\def\sss#1{{\fam\ssfam\relax#1}}

\newfam\clfam
\textfont\clfam=\caligr \scriptfont\clfam=\scaligr \scriptscriptfont\clfam=\sscaligr

\newfam\frfam
\textfont\frfam=\fraktur \scriptfont\frfam=\sfraktur \scriptscriptfont\frfam=\ssfraktur

\def\hpb#1{\setbox0=\hbox{${#1}$}
    \copy0 \kern-\wd0 \kern.2pt \box0}
\def\vpb#1{\setbox0=\hbox{${#1}$}
    \copy0 \kern-\wd0 \raise.08pt \box0}

\def\pmb#1{\setbox0\hbox{${#1}$} \copy0 \kern-\wd0 \kern.2pt \box0}
\def\pmbb#1{\setbox0\hbox{${#1}$} \copy0 \kern-\wd0
      \kern.2pt \copy0 \kern-\wd0 \kern.2pt \box0}
\def\pmbbb#1{\setbox0\hbox{${#1}$} \copy0 \kern-\wd0
      \kern.2pt \copy0 \kern-\wd0 \kern.2pt
    \copy0 \kern-\wd0 \kern.2pt \box0}
\def\pmxb#1{\setbox0\hbox{${#1}$} \copy0 \kern-\wd0
      \kern.2pt \copy0 \kern-\wd0 \kern.2pt
      \copy0 \kern-\wd0 \kern.2pt \copy0 \kern-\wd0 \kern.2pt \box0}
\def\pmxbb#1{\setbox0\hbox{${#1}$} \copy0 \kern-\wd0 \kern.2pt
      \copy0 \kern-\wd0 \kern.2pt
      \copy0 \kern-\wd0 \kern.2pt \copy0 \kern-\wd0 \kern.2pt
      \copy0 \kern-\wd0 \kern.2pt \box0}

\mathchardef\za="710B  
\mathchardef\zb="710C  
\mathchardef\zg="710D  
\mathchardef\zd="710E  
\mathchardef\zve="710F 
\mathchardef\zz="7110  
\mathchardef\zh="7111  
\mathchardef\zvy="7112 
\mathchardef\zi="7113  
\mathchardef\zk="7114  
\mathchardef\zl="7115  
\mathchardef\zm="7116  
\mathchardef\zn="7117  
\mathchardef\zx="7118  
\mathchardef\zp="7119  
\mathchardef\zr="711A  
\mathchardef\zs="711B  
\mathchardef\zt="711C  
\mathchardef\zu="711D  
\mathchardef\zvf="711E 
\mathchardef\zq="711F  
\mathchardef\zc="7120  
\mathchardef\zw="7121  
\mathchardef\ze="7122  
\mathchardef\zy="7123  
\mathchardef\zf="7124  
\mathchardef\zvr="7125 
\mathchardef\zvs="7126 
\mathchardef\zf="7127  
\mathchardef\zG="7000  
\mathchardef\zD="7001  
\mathchardef\zY="7002  
\mathchardef\zL="7003  
\mathchardef\zX="7004  
\mathchardef\zP="7005  
\mathchardef\zS="7006  
\mathchardef\zU="7007  
\mathchardef\zF="7008  
\mathchardef\zW="700A  
\mathchardef\zC="7009  

\newcommand{\be}{\begin{equation}}
\newcommand{\ee}{\end{equation}}

\newcommand{\bea}{\begin{eqnarray}}
\newcommand{\eea}{\end{eqnarray}}
\def\*{{\textstyle *}}
\newcommand{\R}{{\mathbb R}}

\newcommand{\Z}{{\mathbb Z}}

\newcommand{\s}{{\textstyle *}}

\def\Vect{\sss{Vect}}

\def\sT{{\sss T}}

\def\xi{\tx{i}}

\def\s*{{\scriptstyle *}}

\def\cO{\mathcal{O}}

\newcommand{\beas}{\begin{eqnarray*}}
\newcommand{\eeas}{\end{eqnarray*}}

\def\half{\frac{1}{2}}

\title{Modular classes of Q-manifolds: a review and some applications}

   \author{Andrew James Bruce} 
   \address{Mathematics Research Unit, University of Luxembourg, Maison du Nombre 6, avenue de la Fonte, 
L-4364 Esch-sur-Alzette}  
   \email{andrewjamesbruce@googlemail.com}
  
\date{\today}
 
\begin{document}

\begin{abstract}
A Q-manifold is a supermanifold equipped with an odd vector field that squares to zero. The notion of the modular class of a Q-manifold -- which is viewed as the obstruction to the existence of a Q-invariant Berezin volume --  is not well know. We review the basic ideas and then apply this  technology to various examples, including $L_{\infty}$-algebroids and  higher Poisson manifolds. \par
\end{abstract}

 \maketitle

\setcounter{tocdepth}{2}
 \tableofcontents

\section{Introduction} 
The notion of the modular class of a Poisson manifold was first introduced by Koszul \cite{Koszul:1985} without that name, and then reintroduced with that name by  Weinstein  \cite{Weinstein:1997}. The modular class of a Poisson manifold is  understood as the obstruction to the existence of a volume that is invariant under the action of all Hamiltonian vector fields. The closely related notion of the modular class of a Lie algebroid was introduced by  Evens, Lu and Weinstein \cite{Evans:1999}. Recall that the cotangent bundle of a Poisson manifold canonically comes with the structure of a Lie algebroid. It was quickly realised that the modular class of a Poisson manifold and that of its associated cotangent Lie algebroid are the same up to a factor of $2$.  In the other direction,  a Lie algebroid structure on a vector bundle $A$ is equivalent to a linear Poisson structure on the dual $A^{*}$.  The  modular class of a Lie algebroid can then be interpreted as the obstruction to the existence of a  measure on $A^{*}$ that is invariant with respect to the Hamiltonian vector fields on the Poisson manifold $A^{*}$. If the modular class of a Lie algebroid or Poisson manifold is non-zero, then we have a `violation' of the classical Liouville theorem on symplectic manifolds: there is no  volume form that  is constant in the direction of all Hamiltonian vector fields. For a comprehensive review of modular classes of Poisson manifolds and Lie algebroids see Kosmann-Schwarzbach \cite{Kosmann-Schwarzbach:2008}.  \par 
Va\u{\i}ntrob \cite{Vaintrob:1997} provided an elegant description of Lie algebroids in terms of Q-manifolds: that is supermanifolds equipped with an odd vector field that squares to zero. This supermanifold description allows for a very clear definition of the modular class of a Lie algebroid in terms of the divergence of the homological vector field, though naturally the modular class does not depend on the volume form chosen to define the divergence. While this has been known to experts for a while, the only place in print where details can be found is Grabowski \cite{Grabowski:2011,Grabowski:2014}. These notions directly generalise to more general Q-manifolds, this is known to experts such as Kontsevich,\footnote{As described by Grabowski in \cite{Grabowski:2011}. } Lyakhovich \& Sharapov \cite{Lyakhovich:2004}, Roytenberg \cite{Roytenberg:2010}, Voronov \cite{Voronov:2007} etc., but little has actually appeared in the literature. For this reason we think that it will be beneficial to the wider mathematical community to have these notions in print. \par 
In this paper we review the notion of the modular class of a Q-manifold and apply it to several examples. For instance, we present the notion of the modular class of $L_{\infty}$-algebroids (in the $\Z_{2}$-graded setting), particular examples of which include Mehta's Q-algebroids (see \cite{Mehta:2009}). Furthermore, we examine the case of higher Poisson manifolds (cf. \cite{Voronov:2005}).  Associated with any higher Poisson manifold is a suitably superised $L_{\infty}$-algebra on functions on the supermanifold that each bracket satisfies a Leibniz rule. It is not immediately obvious what the notion of a Hamiltonian vector field is in this context. Thus, we cannot directly modify the definition of the modular class of a Poisson manifold to this higher setting: the classical definition is in terms of the divergence of Hamiltonian vector fields. Our solution is to consider a Q-manifold associated to a higher Poisson manifold, in other language the associated $L_{\infty}$-algebroid. We show that the modular class of a higher Poisson manifold, and so be default a Poisson manifold, is tightly related to the BV-Laplacian. As far as we know, this observation has not been made before.\par
Another interesting example is that a double Lie algebroid (cf. Mackenzie \cite{Mackenzie:1992,Mackenzie:2000}). It is known due to  Voronov \cite{Voronov:2012} that a double Lie algebroid is equivalent to a pair of homological vector fields on the total parity reversion of a double vector bundle. We then take the modular class of a double Lie algebroid to be the modular class associated with the sum of the two homological vector fields. To our knowedge the notion of the modular class of a double Lie algebroid has not appeared in the literature. In part this is probabily due to the original definition of Mackenzie being rather complicated.\par 
We must remark that the modular class of a Q-manifold is a characteristic class in the sense that we assign to any Q-manifold a cohomology class, in this case in the standard cohomology. Other examples of characteristic classes of Q-manifolds are discussed in \cite{Kotov:2015,Lyakhovich:2004,Lyakhovich:2010}. The modular class is relatively simple to calculate explcitly for given examples, and so one of the easier to work with characteristic classes.

\medskip
\noindent \textbf{Our use of supermanifolds} We assume that the reader has some familiarity with the basics of the theory of supermanifolds. We will follow the ``Russian school'' and understand a \emph{supermanifold} $M := (|M|, \:  \cO_{M})$ of dimension $n|m$ as a locally superringed space that is locally isomorphic to $\mathbb{R}^{n|m} := \big (\R^{n}, C^{\infty}(\R^{n})\otimes \Lambda(\zx^{1}, \cdots \zx^{m}) \big)$. In particular,  given any point on $|M|$ we can always find a `small enough' open neighbourhood $|U|\subseteq |M|$ such that we can  employ local coordinates $x^{a} := (x^{\mu} , \theta^{i})$ on $M$, where $x^{\mu}$  and $\theta^{i}$ are, respectively, collections of commuting and  anticommuting elements of $\cO_{M}(|U|)$. We will call (global) sections of the structure sheaf \emph{functions}, and often  denote the supercommutative algebra of all functions as $C^{\infty}(M)$. The underlying smooth manifold $|M|$ we refer to as the \emph{reduced manifold}.  We will make heavy use of local coordinates on supermanifolds and employ the standard abuses of notation when it comes to describing morphisms of supermanifolds.\par
The \emph{tangent sheaf} $\mathcal{T}M$ of a supermanifold $M$ is the sheaf of derivations of sections of the structure sheaf -- this is of course a sheaf of locally free $\cO_{M}$-modules. Sections of the tangent sheaf we refer to as \emph{vector fields}, and denote the $\cO_{M}$-module of vector fields as $\Vect(M)$. The total space of the tangent sheaf we will denote by $\sT M$ and refer to this as the \emph{tangent bundle}. By shifting the parity of the fibre coordinates one obtains the \emph{antitangent bundle} $\Pi \sT M$. We will reserve the nomenclature \emph{vector bundle} for the total space of a sheaf of locally free $\cO_{M}$-modules, that is we will be refering to `geometric vector bundles'.\par
There are several good books on the subject of supermanifolds and we recommend  Carmeli, Caston \& Fioresi \cite{Carmeli:2011},  Manin \cite{Manin:1997}  and Varadrajan \cite{Varadrajan:2004}. We will denote the Grassmann parity of an object $A$ by `tilde', i.e., $\widetilde{A} \in \Z_{2}$.  By `even' and `odd' we will be referring the Grassmann parity of  the objects in question. As we will work in the category of smooth supermanifolds, all the algebras, commutators etc. will be $\Z_{2}$-graded.

\section{Divergence operators and Q-manifolds}
\subsection{Berezin forms and volumes}
Let us for simplicity assume that the supermanifolds that we will be dealing with are superoriented (see \cite{Shnader:1988}). That is the underlying reduced manifold will be oriented, and we further require that we have chosen an atlas such that the Jacobian associated to any change of coordinates is strictly positive. The \emph{Berezin bundle} $\Ber(M)$, is understood as the (even) line bundle over $M$ whose sections in a local trivialisation are of the form
$$ \Bs= D[x] s(x),$$ 
where $D[x]$ is the \emph{coordinate volume element}. Under changes of local coordinate  we have
$$D[x'] = D[x] \Ber\left(\frac{\partial x'}{\partial x}  \right).$$
Sections of $\Ber(M)$ are \emph{Berezin forms} on $M$. Note the the Grassmann parity of a Berezin density is determined by $s(x)$. A \emph{Berezin volume} on $M$ is a nowhere vanishing  even Berezin form. We will in the proceeding denote a Berezin volume as $\p = D[x] \rho(x)$. In particular we require that $\rho(x)$ be invertiable. \par 
Any vector field $X \in \Vect(M)$ defines an infinitesimal diffeomorphism, which in local coordinates is 
$$\phi_{\epsilon}^{*}x^{a} = x^{a} + \epsilon X^{a}(x),$$
where $\epsilon$ is an external infinitesimal parameter of parity $\widetilde{\epsilon} = \widetilde{X}$. It is an easy calculation to show that
$$J^{-1} = \Ber\left(\frac{\partial \phi_{\epsilon}^{*}x}{\partial x} \right) = 1 + (-1)^{\widetilde{a}(\widetilde{X}+1)} \epsilon \frac{\partial X^{a}}{\partial x^{a}}.$$
The pullback of a Berezin density is
$$\phi_{\epsilon}^{*}\Bs = D[x] \: J^{-1} \: \phi_{\epsilon}^{*} s(x), $$
and the Lie derivative is thus 
$$L_{X}\Bs = (-1)^{\widetilde{a}(\widetilde{X}+1)}D[x] \left(\frac{\partial X^{a}s }{\partial x^{a}} \right).$$

\subsection{Divergence of a vector field}
In the classical case on a manifold, one needs a volume form (or in the non-oriented case a density) in order to define the divergence of a vector field. The same is true for supermanifolds, and we take the definition  of the divergence of a vector field $X \in \Vect(M)$ with respect to a chosen Berezin volume to be
$$\p \: \textnormal{Div}_{\p}X = L_{X}\p.$$
In local coordinates this definition amounts to
$$\textnormal{Div}_{\p}X = (-1)^{\widetilde{a}(\widetilde{X}+1)} \: \frac{1}{\rho} \frac{\partial}{\partial x^{a}}\left(X^{a} \rho \right).$$
Up to a sign factor, this local expression is exactly the same as the classical case. Moreover, it is not hard to prove the following properties of the divergence
\begin{subequations}
\begin{align}
&\textnormal{Div}_{\p}(f\: X) = f \: \textnormal{Div}_{\p}X + (-1)^{\widetilde{f} \widetilde{X}} X(f);\\
&\textnormal{Div}_{\p'}X = \textnormal{Div}_{\p}X + X(g);\\
&\textnormal{Div}_{\p}[X,Y] = X(\textnormal{Div}_{\p} Y) -(-1)^{\widetilde{X} \widetilde{Y}}Y(\textnormal{Div}_{\p}X);\label{eqn:DivCommutator}
\end{align}
\end{subequations}
where $X$ and $Y \in \Vect(M)$, $f \in C^{\infty}(M)$, and $\p' = \exp(g)\p$ with $g\in C^{\infty}(M)$ is even. These properties, again up to some signs are identical to the properties of the classical divergence operator on a manifold. 
\subsection{Homological vector fields and Q-manifolds}
Q-manifolds  offer  a powerful conceptual formalism to describe many interesting structures in mathematical physics such as Lie algebroids (see \cite{Vaintrob:1997}) and Courant algebroids (see \cite{Roytenberg:2002}). 
\begin{definition}
A \emph{Q-manifold}  is a supermanifold  $M$, equipped with a distinguished  odd vector field $Q\in \Vect(M)$  that `squares to zero', i.e., $Q^{2} = \half [Q,Q] =0$. The vector field $Q$ is referred to as a \emph{homological vector field}, or a \emph{Q-structure}.
\end{definition}
Note that due to extra signs that appear in supergeometry, $[Q,Q] := Q\circ Q + Q \circ Q$, and hence  $Q^2 =0$ is a non-trivial  condition. In local coordinates we have $Q = Q^{a}(x)\frac{\partial}{\partial x^{a}}$, and the condition that $Q$ is homological is 
$$Q^{2} = 0    \Longleftrightarrow     Q^{a}\frac{\partial Q^{b}}{\partial x^{a}} =0.$$
\begin{definition}
Let $(M_{1}, Q_{1})$ and $(M_{2}, Q_{2})$ be Q-manifolds. Then a morphism of supermanifolds $\psi : M_{1} \rightarrow M_{2}$ is a \emph{morphisms of Q-manifolds} if it relates the two  homological vector fields, i.e.,
$$Q_{1} \circ \psi^{*} = \psi^{*} \circ Q_{2}.$$
\end{definition}
To be explicit, let us employ local coordinates $x^{a}$ on $M_{1}$ and $y^{\alpha}$ on $M_{2}$. We will write, using standard abuses of notation $\psi^{*}y^{\alpha} = \psi^{\alpha}(x)$. The statement that $\psi$ be a morphism of Q-manifolds means locally that
$$Q_{1}^{a}(x)\frac{\partial \psi^{\alpha}(x)}{\partial x^{a}} = Q_{2}^{\alpha}(\psi(x)).$$
Evidently, we obtain the category of Q-manifolds via standard compoistion of supermanifold morphisms.  This category also admits products.
\begin{definition}\label{def:productQs}
 Let $(M_{1}, Q_{1})$ and $(M_{2}, Q_{2})$ be Q-manifolds, then their \emph{Q-manifold product} is the  Q-manifold
 $$\big (M_{12} := M_{1}\times M_{2} , \: Q_{12} = Q_{1} + Q_{2} \big ).$$
\end{definition}

\begin{definition}
The \emph{standard cochain complex} associated with an Q-manifold is the  $\Z_{2}$-graded cochain complex $(C^{\infty}(M), Q)$. The resulting cohomology is referred to as the \emph{standard cohomology}  of the Q-manifold.
\end{definition}
We then see that morphisms of Q-manifolds are cochain maps between the respective standard cochain complexes.\par 
\begin{example}
Any supermanifold can be considered as a Q-manifold equipped with the trivial Q-structure $Q=0$. In fact, on manifolds, i.e., pure even supermanifolds, the only possible Q-structure is the trivial one. In this case the resulting standard cohomology is of course also trivial.
\end{example}
\begin{example}
The antitangent bundle of a supermanifold $\Pi \sT M$ comes canonically equipped with a Q-structure, called the de Rham differential. In local coordinates $(x^{a}, \rmd x^{b})$, where $\widetilde{x^{a}} = \widetilde{a}$ and $\widetilde{\rmd x^{b}} = \widetilde{b}+1$, we have $Q := \rmd = \rmd x^{a}\frac{\partial}{\partial x^{a}}$. Differential forms on a supermanifold are understood as functions on $\Pi \sT M$, that is $\Omega^{\bullet}(M) := C^{\infty}(\Pi \sT M)$. The standard cohomology is then just the de Rham cohomology of the supermanifold $M$, which is known to be isomorphic to the de Rham cohomology of the reduced manifold $|M|$.
\end{example}
Associated canonically with any Q-manifold is an `odd anchor' $a_{Q} : M \rightarrow \Pi \sT M$, which is no more than considering the homological vector field as a section of the antitangent bundle. Thus, in local coordinates we have
$$a_{Q}^{*}(x^{a}, \rmd x^{b}) =  (x^{a}, Q^{b}(x)).$$
\begin{proposition}
The `odd anchor' map is a morphism of Q-manifolds between $(M, Q)$ and $(\Pi \sT M, \rmd)$, that is 
$$Q \circ a_{Q}^{*} - a_{Q}^{*} \circ \rmd =0.$$ 
\end{proposition}
\begin{proof}
Via direct computation in local coordinates
\begin{eqnarray*}
Q^{a}\frac{\partial}{\partial x^{a}}(a_{Q}^{*}\omega) - a_{Q}^{*}\left(\rmd x^{a}\frac{\partial \omega}{\partial x^{a}} \right) &= &   Q^{a}a_{Q}^{*}\left( \frac{\partial \omega}{\partial x^{a}}\right) - Q^{a}a_{Q}^{*}\left( \frac{\partial \omega}{\partial x^{a}}\right) + Q^{a}\frac{\partial Q^{b}}{\partial x^{a}} a_{Q}^{*}\left( \frac{\partial \omega}{\partial \rmd x^{b}} \right)\\ &= &  Q^{a}\frac{\partial Q^{b}}{\partial x^{a}} a_{Q}^{*}\left( \frac{\partial \omega}{\partial \rmd x^{b}} \right).
\end{eqnarray*}
As $Q^{2} =0$ we obtain the desired result.
\end{proof}
\begin{definition}
Let $(M,Q)$ be a Q-manifold. A vector field $X \in \Vect(M)$ is said to be an \emph{infinitesimal symmetry} or just a \emph{symmetry} of a Q-manifold if and only if $[X,Q] =0$. A symmetry $X$ is said to be an \emph{inner symmetry} if there exists another vector field $Y \in \Vect(M)$ such that $X =[Q,Y]$.
\end{definition}
\begin{example}
On any Q-manifold $(M,Q)$, the homological vector field $Q$ is rather trivially a symmetry.
\end{example}
\begin{example}
Consider the Q-manifold $(\Pi \sT M, \rmd)$. Any vector field $X \in \Vect(M)$ can be sent to its interior derivative $X \mapsto i_{X}$, which in local coordinates is given by
$$i_{X} =  (-1)^{\widetilde{X}}X^{a}(x)\frac{\partial}{\partial \rmd x^{a}}.$$
The Lie derivative, $L_{X} := [\rmd, i_{X}]$, is an inner symmetry of $(\Pi \sT M, \rmd)$, i.e., $[\rmd, L_{X}] =0$.
\end{example}
\section{Modular classes}
\subsection{Definition and main properties of the modular class}
The modular class of a Q-manifold is defined in terms of the divergence of the homological vector field.  Before we give the definition, we need a couple of observations.
\begin{proposition}
Let $(M,Q)$ be a Q-manifold, then $\textnormal{Div}_{\p}Q$ is Q-closed, i.e., $Q\left(\textnormal{Div}_{\p}Q\right) =0$.
\end{proposition}
\begin{proof}As $[L_{Q}, L_{Q}] = L_{[Q,Q]} =0$ we have that
$$ L_{Q}\left(L_{Q} \p \right)   = L_{Q}(\p \textnormal{Div}_{\p}Q) = \p \left(\textnormal{Div}_{\p}Q \right)^{2} + \p \: Q\left(\textnormal{Div}_{\p}Q\right) =0.$$
As $\textnormal{Div}_{\p}Q$ is a Grassmann odd function on $M$, it follows that $\left(\textnormal{Div}_{\p}Q \right)^{2}=0$, and we obtain the desired result.
\end{proof}
\begin{proposition}
Let $(M,Q)$ be a Q-manifold, then the derivative of the divergence of the homological vector field $Q$ in the direction of any symmetry is $Q$-exact.
\end{proposition}
\begin{proof}
It follows from the definition of a symmetry, i.e., $[X,Q]=0$, and  the properties of the divergence (\ref{eqn:DivCommutator}) that
$$X\big(\textnormal{Div}_{\p}Q\big) =  (-1)^{\widetilde{X}} Q\big(\textnormal{Div}_{\p}X  \big).$$
\end{proof}
From the properties of the divergence, it is clear that if we change the Berezin volume the the divergence changes by a Q-exact term, i.e. $\textnormal{Div}_{\p'}Q = \textnormal{Div}_{\p}Q + Q(g)$, where $\p' = \exp(g)\p$. Similarly, any (small) change in the divergance in the direction of a symmetry is $Q$-exact.   We then have the following definition
\begin{definition}
The \emph{modular class} of a Q-manifold  is the standard cohomology class of $\textnormal{Div}_{\p}Q$, i.e., 
$$\textnormal{Mod}(Q) := [\textnormal{Div}_{\p}Q]_{\textnormal{St}}.$$
\end{definition}
Note that the modular class is independent of any chosen Berezin volume as any other choice of volume leads to divergences that differ only by something Q-exact, and so Q-closed. Thus, the modular class is a characteristic class of a Q-manifold.  The vanishing of the modular class is a necessary and sufficient condition for the existence of a Berezin volume that is Q-invariant, that is for some choice of Berezin volume $\p$ we have that $L_{Q}\p =0$. \par 
In some given set of local coordinates one can write out the divergence,
$$\textnormal{Div}_{\p}Q =  \frac{\partial Q^{a}}{\partial x^{a}} + Q(\log(\rho)).$$
The \emph{local (characteristic) representative} of the modular class is understood as just the term
\begin{equation}\label{eqn:localrep}
\phi_{Q}(x) := \frac{\partial Q^{a}}{\partial x^{a}}(x).
\end{equation}
In general this term is \emph{not} invariant under changes of coordinates, only the full expression for the divergence is.  However, as we are always dropping terms that are Q-exact, the local representative is still meaningful, though as written it is only a local function on $M$.
\begin{remark}
The expression (\ref{eqn:localrep}) gives the local representative of the standard (coordinate) density (in some chosen local coordinates). In general we do not have the Poincar\'{e} lemma: meaning that Q-closed functions are \emph{not} necessarily locally Q-exact. Thus, it makes sense to speak of a local  (characteristic) representative of the modular class.
\end{remark}
\begin{definition}
A Q-manifold $(M,Q)$ is said to be a \emph{unimodular Q-manifold} if its modular class vanishes. In other words, if there exists a Q-invariant Berezinian volume.
\end{definition}
\begin{example}
The Q-manifold $(\Pi \sT M, \rmd)$ comes with a canonical Berezin volume, which in local coordinates is just $D[x,dx]$, and clearly this is invariant with respect to the de Rham differential. Thus, we have a unimodular Q-manifold.
\end{example}
\begin{example}
Clearly the modular class is explicitly dependent on the homological vector field under study. In particular, we can equip the supermanifold  $\Pi \sT M$ with homological vector fields other than the canonical de Rham differential. For example, let us take $M$ to be a manifold (this  suppresses some signs) and equip it with the following odd vector field
$$Q = \rmd x^{b}\mathcal{N}_{b}^{a}(x) \frac{\partial}{\partial x^{a}} + \frac{1}{2} \rmd x^{a} \rmd x^{b} \left(\frac{\partial \mathcal{N}_{a}^{c}}{\partial x^{b}}  {-}  \frac{\partial \mathcal{N}_{b}^{c}}{\partial x^{a}}\right)\frac{\partial }{ \partial \rmd x^{c}}.$$ 
If the 1-1--tensor $\mathcal{N}_{a}^{b}$ is a \emph{Nijenhuis tensor}, then $Q$ is a homological vector field. Thus, $(\Pi \sT M, Q )$ is in fact a Lie algebroid, with the local representative of modular class being
$$\phi_{\mathcal{N}}(x, \rmd x) =  \rmd x^{a}\frac{\partial \mathcal{N}_{b}^{b}}{\partial x^{a}} =  \rmd \tr(\mathcal{N} ),$$
which in general is non-vanishing. This example is in agreement with Damianou \&   Fernandes \cite[Proposition 2.4]{Damianou:2008}.
\end{example}
\begin{example}
Let  $(M,Q)$ be a Q-manifold, then clearly $(\Pi \sT M, L_{Q})$ is also a Q-manifold. In natural coordinates $(x^{a}, \rmd x^{b})$ on $\Pi \sT M$ the Lie derivative is given by
$$L_{Q} = [\rmd , i_{Q}] = Q^{a} \frac{\partial}{\partial x^{a}}\:  {-} \:  \rmd x^{b}\frac{\partial Q^{a}}{\partial x^{b}}\frac{\partial }{\partial \rmd x^{a}}.$$
Then via inspect we see that $\phi_{L_{Q}} =0$, and so we have a unimodular Q-manifold.
\end{example}
\begin{example}
Combining the two previous example, as $[\rmd, L_{Q}] =0$ we see that $(\Pi \sT M , L_{Q} + \rmd)$ is also a unimodular Q-manifold.
\end{example}
\begin{remark}
The  Mathai--Quillen--Kalkman isomorphism tells us that 
$$\rmd + L_{Q} = \rme^{- i_{Q}} \rmd \rme^{ i_{Q}},$$
and so it is not suprising that we obtain a unimodular Q-manifold in the previous example.
\end{remark}
\begin{example}
Let  $(M,Q)$ be a Q-manifold, then  $(\sT^{*} M, L_{Q})$ is also a Q-manifold. In natural coordinates $(x^{a}, p_{b})$ the canonical Poisson bracket is given by
$$\{F,G   \} = (-1)^{\widetilde{a}(\widetilde{F} +1)}\frac{\partial F}{\partial p_{a}} \frac{\partial G}{\partial x^{a}} \:{-}\: (-1)^{\widetilde{a} \widetilde{F}}\frac{\partial F}{\partial x^{a}} \frac{\partial G}{\partial p_{a}},$$
for any $F$ and $G \in C^{\infty}(\sT^{*}M)$. The Lie derivative can then be understood as
$$L_{Q} :=  \{S, \bullet  \},$$
where $S = Q^{a}(x)p_{a}$ is the symbol of the homological vector field $Q$. Thus, the Lie derivative is a Hamiltonian vector field.  Explcitly we have
$$L_{Q} = Q^{a}\frac{\partial}{\partial x^{a}} - (-1)^{\widetilde{a}} \frac{\partial Q^{b}}{\partial x^{a}}p_{b} \frac{\partial }{\partial p_{a}}.$$ 
Via inspection we see that $\phi_{L_{Q}} =0$, and so we have a unimodular Q-manifold.  In fact this is not at all unexpected as we have a version of Liouville's theorem on symplectic supermanifolds: there is always a Berezin volume on any \emph{even} symplectic supermanifold that us invariant with respect to all Hamiltonian vector fields.
\end{example}
\begin{example}
Let  $(M,Q)$ be a Q-manifold, then  $(\Pi \sT^{*} M, L_{Q})$ is also a Q-manifold. Similarly to the previous example we have a canonical Schouten (odd Poisson) bracket and we define the Lie derivative as $L_{Q} :=  \SN{\mathcal{P}, \bullet}$, where $\mathcal{P}$ is now the odd symbol. As $\SN{\mathcal{P}, \mathcal{P}} =0$, it is clear that $L_{Q}$ is a homological vector field. We can think a Q-structure as a one-Poisson structure. However, in this case we do not have a generalisation of Liouville's theorem and so in general we do \emph{not} have a unimodular Q-manifold.  In fact, direct calculation yields
$$\phi_{L_{Q}} = 2\: \phi_{Q},$$
and so $(\Pi \sT^{*} M, L_{Q})$ is unimodular when $(M,Q)$ is unimodular. We will return to similar examples in Subsection \ref{subsec:HigherPoisson} where we discuss higher Poisson manifolds.
\end{example}
The modular class behaves additively under the Q-manifold product (see Definition \ref{def:productQs}). A little more carefully, we have the following.
\begin{proposition}
Let $(M_{1}, Q_{1})$ and $(M_{2}, Q_{2})$ be Q-manifolds, then the modular class of their Q-manifold product is addative, in the following sense:
$$ \textnormal{Mod}(Q_{12}) =  \textnormal{Mod}(Q_{1}) + \textnormal{Mod}(Q_{2}).$$
\end{proposition} 
\begin{proof}
This follows from the linear properties of the divergence, and the fact that a Berezin volume on $M_{12}$ is given by the product of Berezin volumes on $M_{1}$ and $M_{2}$. 
\end{proof}

\subsection{Poincar\'{e} duality}
Recall that the standard cochain complex of a Q-manifold is $(C^{\infty}(M),Q)$. We define the \emph{standard chain complex} as $ \big ( \textnormal{Vol}(M), L_{Q} \big)$, where $\textnormal{Vol}(M)$  stands for the volume forms on $M$, and $L_{Q}$ is the Lie derivative along $Q$. Assuming that $M$ is superoriented and compact (otherwise one should consider compactly  supported densities), we have a natural pairing of a volume with a function via integration,
$$\langle \p , f  \rangle =  \int_{M} \p \: f. $$
From the basic properties of the Lie derivative and the integral we see that
$$\int_{M} L_{Q}(\p \: f)   = \int_{M} \left(L_{Q}\p \right)\: f  + \int_{M} \p \: Q(f).$$
If $(M, Q)$ is unimodular, then intgration is invariant under the action of $Q$, thus the left hand side of the above vanishes. In this case, integeration gives a natural isomorphism between  the standard cochain complex and the standard chain complex, i.e.,
$$\langle L_{Q}\p, f \rangle = {-} \langle  \p, Q(f) \rangle. $$
This natural isomorphism is the generalisation of classical Poincar\'{e} duality.

\subsection{Relative modular classes}
The relative modular class, or the modular class of a Q-manifold morphism can directly be defined as follows.
\begin{definition}
Given a Q-manifold morphism $\psi : (M_{1}, Q_{1}) \rightarrow (M_{2}, Q_{2})$,  the \emph{relative modular class} of $\psi$ is
$$\textnormal{Mod}(\psi) :=  \textnormal{Mod}(Q_{1}) {-}  \psi^{*}(\textnormal{Mod}(Q_{2})),$$
which is a standard cohomology class of $(M_{1}, Q_{1})$.
\end{definition}
By definition, the relative modular class measures the failure of a morpihsm of Q-manifolds to preserve the modular class. In general, there is no reason to expect the modular class to be preserved under morphisms.
\begin{example}
Let $(M, Q)$  be a Q-manifold and  let $a_{Q} : M \rightarrow \Pi \sT M$ be the associated `odd anchor'. Then it is clear that
$$\textnormal{Mod}(a_{Q}) = \textnormal{Mod}(Q).$$
\end{example}
In light of the above example, we see that the modular class of a Q-manifold is a `universal relative modular class' in the sense that it is a canonical relative modular class associated with any Q-manifold. 
\begin{example}
The relative modular class of a Q-manifold isomorphism vanishes.
\end{example}
\begin{example}  Let $\textnormal{j} : N \rightarrow M$ be a subsupermanifold. Furthermore suppose that both these supermaifolds are Q-manifolds, and that we have a morphism of Q-manifolds given by the inclusion morphism $\textnormal{j}$. That is, the associated restriction map satisfies
$$\textnormal{j}^{*} \circ Q_{M} = Q_{N} \circ \textnormal{j}^{*}.$$
As we have a subsupermanifold, we can always find adapted coordinates $(x^{a}, y^{\alpha})$ on $M$, such that the restriction map is given by $\textnormal{j}^{*}(x^{a}, y^{\alpha}) =  (x^{a}, 0)$. Then in these adapted local coordinates we have
\begin{align*}
& Q_{M} = Q_{M}^{a}(x,y)\frac{\partial}{\partial x^{a}} + Q_{M}^{\alpha}(x,y)\frac{\partial}{\partial y^{\alpha}},\\
& Q_{N} = Q_{N}^{a}(x)\frac{\partial}{\partial x^{a}},
\end{align*}
and the condition that these be $\textnormal{j}$-related means
\begin{align*}
&  Q_{N}^{a}(x) = Q_{M}^{a}(x,0), &\textnormal{and}&&  Q_{M}^{\alpha}(x,0) =0.
\end{align*}
As $\textnormal{j}$ is a morphism of $Q$-manifolds we have an induced map between the modular classes
$$\textnormal{j}^{*} :  \textnormal{Mod}(Q_{M}) \longrightarrow \textnormal{Mod}(Q_{N}).$$
However, in general $\textnormal{Mod}(Q_{N}) \neq \textnormal{j}^{*}\textnormal{Mod}(Q_{M})$. By definition, the difference is the relative modular class of $\textnormal{j}$. Directly,  the local representative is
$$\phi_{\textnormal{j}}(x) = {-} \left.\frac{\partial Q^{\alpha}_{M}}{\partial y^{\alpha}}(x,y)\right |_{y=0}.$$ 
Thus we see that the relative modular class only depends on the \emph{linear behaviour} of $Q_{M}^{\alpha}(x,y)$ near $y=0$. Let us then Talyor expand `near' $N$ (noting that on $N$  we have $Q_{M}^\alpha =0 $ ) 
$$Q_{M}^{\alpha}(x,y) =    y^{\beta}\mathbb{A}_{\beta}^{\:\: \alpha}(x) + \cO(y^2),$$
where we have defined 
$$\mathbb{A}_{\beta}^{\:\: \alpha}(x) :=  \left .\frac{\partial Q_{M}^{\alpha}}{\partial y^{\beta}}(x,y) \right|_{y=0}.$$
Note  $\widetilde{\mathbb{A}_{\beta}^{\:\: \alpha}} = \widetilde{\alpha} +\widetilde{\beta} +1$, and so is an \emph{odd matrix} -- this effects the definition of the supertrace.  With this notation in place, we can write
$$\phi_{\textnormal{j}}(x) = {-} \str(\mathbb{A}).$$
\end{example}

\section{Applications and examples}
\subsection{$L_{\infty}$-algebroids}
We follow Bruce \cite{Bruce:2011} (also see \cite{Khudaverdian:2008}) in our definition of a $L_{\infty}$-algebroid,  slightly different notions with different gradings appear in the literature (see for example \cite{Bonavolonta:2013,Sheng:2016}).
\begin{definition}
A super vector bundle $A$ is said to be a \emph{$L_{\infty}$-algebroid} if there exists a homological vector field $\rmd_{A} \in \Vect(\Pi A)$. 
\end{definition}
Note that we do not insist that the homological vector field be linear, or in the graded language, be of degree one with respect to the natural $\mathbb{N}$-grading induced by declaring the base coordinates to be of degree zero and the fibre coordinates to be of degree one.  If the homological vector field is of degree one then we recover via Va\u{\i}ntrob \cite{Vaintrob:1997} a (super) Lie algebroid.\par 
In natural local coordinates $(x^{a}, \zx^{\alpha})$ on $\Pi A$, the Q-structure is of the form
$$ \rmd_{A} =  \sum_{n=0}^{\infty}\frac{1}{n!} \zx^{\alpha_{1}} \cdots \zx^{\alpha_{n}}Q^{a}_{\alpha_{n} \cdots \alpha_{1}}(x) \frac{\partial}{\partial x^{a}}\:  +  \:  \sum_{n=0}^{\infty}\frac{1}{n!} \zx^{\alpha_{1}} \cdots \zx^{\alpha_{n}}Q^{\beta}_{\alpha_{n} \cdots \alpha_{1}}(x) \frac{\partial}{\partial \zx^{\beta}}. $$
The local representative of the modular class is thus
$$\phi_{\rmd_{A}}(x, \zx) =   \sum_{n=0}^{\infty}\frac{(-1)^{\epsilon}}{n!}   \zx^{\alpha_{1}} \cdots \zx^{\alpha_{n}} \frac{\partial Q^{a}_{\alpha_{n} \cdots \alpha_{1}}}{\partial x^{a} \hfill}(x)  \: +  \:   \sum_{n=1}^{\infty}\frac{1}{(n-1)!} \zx^{\alpha_{1}} \cdots \zx^{\alpha_{n}}Q^{\beta}_{\alpha_{n} \cdots \alpha_{1} \beta}(x),$$
where the sign factor is $\epsilon =  \widetilde{a}(\widetilde{\alpha_{1}} + \cdots  + \widetilde{\alpha_{n}} + n)$.\par 
Restricting attention to Lie algebroids we obtain
$$\phi_{\rmd_{A}}(x, \zx) = \zx^{\alpha} \left ( (-1)^{\widetilde{a}(\widetilde{\alpha} +1)} \frac{\partial Q_{\alpha}^{a}}{\partial x^{a}}(x)   \: +\: Q_{\alpha \beta}^{\beta}(x)\right ),$$
which is in agreement with the classical literature, e.g. \cite{Grabowski:2006}. For the case of Lie algebroids, one speaks of a \emph{characteristic local section} of $A^{*}$ as the local representative of the modular class is linear in $\zx$. However, for the case of $L_{\infty}$-algebroids we have an inhomogeneous \emph{characteristic local $A$-form}.\par 
We will say that a $L_{\infty}$-algebroid is a \emph{unimodular $L_{\infty}$-algebroid} if its modular class vanishes. We can think of $L_{\infty}$-algebras (in the $\Z_{2}$-graded conventions) as $L_{\infty}$-algebroids over a point. The definition of  a unimodular $L_{\infty}$-algebra is clear, and  coinsides with the definition given by Gran{\aa}ker \cite{Granaker:2008}, also see Braun \&  Lazarev \cite{Braun:2015}. Further restricting attention to Lie algebras (super or not) reproduces the notion of a unimodular Lie algebra --  the  adjoint map is trace free for all elements in the Lie algebra. Examples of  unimodular Lie algerbas include all Abelian Lie algebras, the Heisenberg Lie algebra and nilpotent Lie algebras.
\begin{remark}
A weighted Lie algebroid is  a Lie algebroid equipped with a homogeneity structure such that $\rmh_{t}: \Pi A \rightarrow \Pi A$ is a morphism of Lie algebroids for all $t \in \R$ (see \cite{Bruce:2016}).  As the homological vector field encoding a weighted Lie algebroid structure is of weight $(0,1)$ (this follows from the definition) the notion of $\rmh$-homogeneous cochains and coboundaries makes sense, thus the standard cohomology inherits a further $\mathbb{N}$-graded structure. From the definition of the modular class we see that $[\textnormal{Div}_{\p}Q]_{\textnormal{St}}$ is homogeneous and of degree zero. This covers the example of $\mathcal{VB}$-algebroids where the homogeneity structure is regular. 
\end{remark}
\begin{example}
Mehta \cite{Mehta:2009} defines a Q-algebroid as a Lie algebroid $(\Pi A, \rmd_{A})$, equipped with a morphic weight zero homological vector field $\Xi \in \Vect(\Pi A)$  (following Mehta's original notation). By morphic, we mean that we have a symmetry, i.e., $[\rmd_{A}, \Xi] =0$. As we have two commuting homological vector fields, we can add them to obtain another homological vector field which is inhomogeneous in weight -- thus we have an $L_{\infty}$-algebroid.  In natural local coordinates we have
$$Q :=  \rmd_{A} + \Xi = \left( Q^{a}(x) + \zx^{\alpha}Q_{\alpha}^{a}(x)  \right)\frac{\partial}{\partial x^{a}}  + \left( \zx^{\alpha}Q_{\alpha}^{\gamma}(x) + \frac{1}{2!} \zx^{\alpha} \zx^{\beta}Q_{\beta \alpha}^{\gamma}(x) \right)\frac{\partial}{\partial \zx^{\gamma}}.$$
The local representative of the modular class is thus
$$\phi_{Q}(x, \zx) = \left( \frac{\partial Q^{a}}{\partial x^{a}}(x)  + Q_{\alpha}^{\alpha}(x) \right)   +  \zx^{\alpha} \left ( (-1)^{\widetilde{a}(\widetilde{\alpha} +1)} \frac{\partial Q_{\alpha}^{a}}{\partial x^{a}} (x)+ Q^{\beta}_{\alpha \beta}(x) \right).$$
\end{example}

\subsection{Higher Poisson manifolds}\label{subsec:HigherPoisson}
It is well known that the modular class of a Poisson manifold is half that of the modular class of the associated cotangent Lie algebroid. We can then use this fact to \emph{define} the modular class of a higher Poisson manifold. First let us recall the definition of a higher Poisson manifold (cf. \cite{Khudaverdian:2008,Voronov:2005})
\begin{definition}
A \emph{higher Poisson manifold} is a pair $(M, \mathcal{P})$, where $\mathcal{P} \in C^{\infty}(\Pi \sT^{*}M)$ is an even (pseudo)multivector field, known as a \emph{homotopy Poisson structure}, that satisfies the Poisson condition $\SN{\mathcal{P}, \mathcal{P}} =0$, where the bracket is the canonical Schouten bracket on $\Pi \sT^{*}M$.
\end{definition}
It is clear, due to the Poisson condition that $(\Pi \sT^{*}M ,\:  Q_{\mathcal{P}} = \SN{\mathcal{P}, \bullet} )$ is a Q-manifold. In fact, we have an $L_{\infty}$-algebroid (see \cite{Bruce:2011}). We then define the modular class of a higher Poisson manifold to be half that of the modular class of the  associated $L_{\infty}$-algebroid -- doing so means that we cover the classical case precisely.\par
Let us examine the local representative of the modular class of $(\Pi \sT^{*}M ,\:  Q_{\mathcal{P}})$ in Darboux coordinates $(x^{a}, x^{*}_{b})$. The canonical Schouten bracket in these coordinates is given by
$$\SN{F,G} = (-1)^{(\widetilde{a}+1)(\widetilde{F}+1)} \frac{\partial F}{\partial x^{*}_{a}} \frac{\partial G}{\partial x^{a}} \;  {-} \;  (-1)^{\widetilde{a}(\widetilde{F}+1)}\frac{\partial F}{\partial x^{a}} \frac{\partial G}{\partial x^{*}_{a}},  $$
 for any $F$ and $G \in C^{\infty}(\Pi \sT^{*}M)$. Direct computations shows that
 $$Q_{\mathcal{P}} = (-1)^{\widetilde{a}+1 } \left(\frac{\partial \mathcal{P}}{\partial x^{*}_{a}} \frac{\partial}{\partial x^{a}} + \frac{\partial \mathcal{P}}{\partial x^{a}}\frac{\partial}{\partial x^{*}_{a}}  \right).$$
 Thus, following the definitions and a simple reordering of the derivatives we obtain
 $$\phi_{Q_\mathcal{P}}(x, x^{*}) = (-1)^{\widetilde{a}+1} \: 2 \left(\frac{\partial^{2}\mathcal{P}}{\partial x^{a} \partial x^{*}_{a}} \right),$$
which up to a factor of $2$ is the (finite dimensional) BV-Laplacian acting on functions (cf. \cite{Batalin:1981,Batalin:1983,Batalin:1984}). In hindsight this is not unexpected as the BV-Laplacian can be naturally identified with the divergance of a multivector field with respect to the coordinate volume (see \cite{Khudaveridan:2002a} for a discussion of this). Via these considerations we are led to the following.
\begin{theorem}
The modular class of a higher Poisson manifold is the standard cohomology class of BV-Laplacian acting on the homotopy Poisson structure, i.e., the cohomology class of
$$\Delta_{\p}\mathcal{P} := \frac{1}{2}\textnormal{Div}_{\p}\: Q_{\mathcal{P}} = (-1)^{\widetilde{a}+1} \: \frac{\partial^{2}\mathcal{P}}{\partial x^{a} \partial x^{*}_{a}}  + \SN{\mathcal{P}, \log(\sqrt{\rho})}.$$ 
\end{theorem}
The vanishing of the modular class implies that there exists a Berezin volume on $\Pi \sT^{*}M$ that is $Q_{\mathcal{P}}$-invariant. This has an infinite dimensional analogue in the BV-formalism. Namely, if we `interpret' a homotopy Poisson structure $\mathcal{P}$ to be be a classical extended action, i.e., $x=$  `fields + ghosts' and $x^{*}=$ `antifields + antighosts', then $Q_{\mathcal{P}}$ is the BRST-operator. We ignore the additional gradings of ghost number etc. The Poisson condition $\SN{\mathcal{P}, \mathcal{P}}=0$ is the analogue of the classical master equation. If the modular class of the extended classical action vanishes then there exists a path integral measure that is BRST-invariant. As the (exponential of the) extended action plus sources is BRST-invariant, the path integral itself is BRST-invariant. One must of course take these statements with a ``grain of salt'' as things are not so well defined in the infinite dimensional setting of quantum field theory. \par 
\begin{remark}
Thinking in terms of quantum field theory, the vanishing of the modular class implies that there is no `one-loop anomaly' (cf. \cite{Roytenberg:2010}). Let us  add `loop corrections', i.e.,  $ \mathcal{P} \rightsquigarrow \mathcal{P}[[\hbar]]  := \mathcal{P} + \hbar \mathcal{P}_{1} + \mathcal{O}(\hbar^{2})$ and insist that the \emph{quantum master equation} is satisfied
$$\Delta \:\rme^{\frac{\rmi}{\hbar} \mathcal{P}[[\hbar]]} =0.$$
The order zero term is just $\SN{\mathcal{P}, \mathcal{P}} =0$,  and the order $\hbar$ term is $\rmi \: \Delta \mathcal{P} = \SN{\mathcal{P}, \mathcal{P}_{1}}$, where we neglect  the choice of density.  Thus, if the modular class vanishes then we can consistantly find up to first order in $\hbar$ a `quantum action', i.e., there is no one-loop BV anomaly.
\end{remark}

\begin{remark}
Similarly, one can consider a homotopy  Schouten structure as an odd analogue of a homotopy Poisson structure, i.e., $S \in C^{\infty}(\sT^{*}M)$, $\widetilde{S} = 1$ and $\{ S,S\}=0$, where the bracket is now the canonical Poisson bracket on the cotangent bundle of a supermanifold $M$. However, as we are dealing with even symplectic geometry we always have the  Liouville volume  which is invariant under the action of Hamiltonian vector fields. Thus, the modular class of the Q-manifold $(\sT^{*}M, \:  Q = \{S,\bullet \})$ is zero.  As we have defined it, the modular class of any higher Schouten manifold  always vanishes. A different notion of the modular class of a Schouten (odd Poisson) manifold can be found in the work of Khudaverdian \& Voronov \cite{Khudaverdian:2002}. Similarly, Courant algebroids understood as even symplectic $NQ$-manifolds of degree $2$ via Roytenberg \cite{Roytenberg:2002}, have vanishing modular class. 
\end{remark}
\begin{example}
If we consider a Poisson supermanifold, i.e., we have a Poisson structure $\mathcal{P} = \frac{1}{2} \mathcal{P}^{ab}(x)x^{*}_{b}x^{*}_{a}$, then the local representative of the modular class is given by
$$\phi_{\mathcal{P}} = \left(\frac{\partial \mathcal{P}^{ab}}{\partial x^{a}}  \right)x^{*}_{b},$$
in agreement with with classical case (see \cite{Weinstein:1997}).
\end{example}
\begin{example}
Consider an order three higher Poisson structure
$$\mathcal{P} = \mathcal{P}(x) + \mathcal{P}^{a}(x)x^{*}_{a} + \frac{1}{2!}\mathcal{P}^{ab}(x)x^{*}_{b} x^{*}_{a} + \frac{1}{2!}\mathcal{P}^{abc}(x)x^{*}_{c} x^{*}_{b} x^{*}_{a}.$$
Then the local representative of the modular class is
$$\phi_{\mathcal{P}} =  \left(\frac{\partial \mathcal{P}^{a}}{\partial x^{a}}\right)   +   \left(\frac{\partial \mathcal{P}^{ab}}{\partial x^{a}}  \right)x^{*}_{b}+  \frac{1}{2!} \left(\frac{\partial \mathcal{P}^{abc}}{\partial x^{a}}  \right)x^{*}_{c} x^{*}_{b}.$$
Note that the order zero piece $\mathcal{P}(x)$ does not contribute to the local representative of the modular class. Thus, if we consiser a zero-Poisson structure (just some chosen even function on $M$), the the modular class always vanishes. Furthermore, note that for the local representative to vanish, we require that each component by order in antimomenta $x^{*}$ to vanish.
\end{example}
\subsection{Double Lie algebroids}
Following Voronov \cite{Voronov:2007,Voronov:2012} we take the following definition.
\begin{definition}
A double vector bundle $D$ is said to be a \emph{double Lie algebroid} if the total parity reversed double vector bundle $\Pi^{2}D$ comes equipped with a pair of commuting homological vector fields $Q_{(0,1)}$ and  $Q_{(1,0)}$ of bi-weight $(0,1)$ and $(1,0)$, respectively.
\end{definition}
Up to a natural isomorphism, it does not matter if we take first shift the parity in  the vertical and then horizontal directions, or vice versa. 
\begin{center}
\leavevmode
\begin{xy}
(0,20)*+{\Pi^{2}D}="a"; (30,20)*+{\Pi B}="b";%
(0,0)*+{\Pi A}="c"; (30,0)*+{ M}="d";%
{\ar "a";"b"}?*!/_3mm/{\pi_{DB}};
{\ar "a";"c"}?*!/^5mm/{\pi_{DA}};
{\ar "b";"d"}?*!/_5mm/{\pi_{B}};
{\ar "c";"d"}?*!/^3mm/{\pi_{A}};
\end{xy}
\end{center}
For concreteness we can make the choice $\Pi^{2} = \Pi_{B} \Pi_{A}$. Let us employ homogeneous local coordinates 
$$(\underbrace{x^{a}}_{(0,0)}, \: \underbrace{\zx^{\alpha}}_{(0,1)} ,\: \underbrace{\theta^{i}}_{(1,0)},\: \underbrace{z^{\mu}}_{(1,1)}),$$
where $(x^{a}, \zx^{\alpha})$ form a coordinate system on $\Pi A$, and $(x^{a}, \theta^{i})$ form a coordinate system on $\Pi B$. In these homogeneous coordinates the pair of homological vector fields are given by
\begin{eqnarray*}
Q_{(0,1)} &=& \zx^{\alpha}Q_{\alpha}^{a}(x)\frac{\partial}{\partial x^{a}} +\frac{1}{2!} \zx^{\alpha}\zx^{\beta}Q_{\beta \alpha}^{\gamma}(x)\frac{\partial}{\partial \zx^{\gamma}}\\
 &+& \left(z^{\mu}Q_{\mu}^{\:\: i}(x) + \theta^{j}\zx^{\alpha}(Q_{\alpha})_{j}^{\:\: i}(x) \right)\frac{\partial}{\partial \theta^{i}} + \left(z^{\nu}\zx^{\alpha}(Q_{\alpha})_{\nu}^{\:\: \mu}(x)   + \frac{1}{2!} \theta^{i} \zx^{\alpha}\zx^{\beta}(Q_{\beta \alpha})_{i}^{\:\:\: \mu}(x) \right)\frac{\partial}{\partial z^{\mu}},\\
 Q_{(1,0)} &=& \theta^{i}Q_{i}^{a}(x)\frac{\partial}{\partial x^{a}} +\frac{1}{2!} \theta^{i}\theta^{j}Q_{j i}^{k}(x)\frac{\partial}{\partial \theta^{k}}\\
 &+& \left(z^{\mu}Q_{\mu}^{\:\: \alpha}(x) + \zx^{\beta} \theta^{i}(Q_{i})_{\beta}^{\:\: \alpha}(x) \right)\frac{\partial}{\partial \zx^{\alpha}} + \left(z^{\nu}\theta^{i}(Q_{i})_{\nu}^{\:\: \mu}(x)   + \frac{1}{2!} \zx^{\alpha} \theta^{i}\theta^{j}(Q_{ji})_{\alpha}^{\:\:\: \mu}(x) \right)\frac{\partial}{\partial z^{\mu}}.
\end{eqnarray*}
Because the two homological vector fields commute, their sum $Q = Q_{(0,1)} + Q_{(1,0)}$ is also a homological vector field. We can then define the modular class of a double Lie algebroid as the modular class of the Q-manifold $(\Pi^{2}D, Q)$. In doing so we see that (up to Grassmann parity) the local representative is a local section of $A^{*} \oplus B^{*}$. Explicitly in local coordinates we have
\begin{eqnarray*}
\phi_{Q}(x, \zx, \theta) &=& \zx^{\alpha} \left((-1)^{\widetilde{a}(\widetilde{\alpha}+1)} \frac{\partial Q_{\alpha}^{a}}{\partial x^{a}} + Q^{\beta}_{\alpha \beta} + (Q_{\alpha})_{i}^{\:\: i} + (Q_{\alpha})_{\mu}^{\:\: \mu} \right)\\
&+& \theta^{i} \left((-1)^{\widetilde{a}(\widetilde{i}+1)} \frac{\partial Q_{i}^{a}}{\partial x^{a}} + Q^{j}_{i j} + (Q_{i})_{\alpha}^{\:\: \alpha} + (Q_{i})_{\mu}^{\:\: \mu} \right).
\end{eqnarray*}
As a specific example consider a Lie algebroid $(\Pi A, \rmd_{A})$ and its antitangent bundle $\Pi \sT \Pi A$ which naturally comes equipped with with two commuting homological vector fields $Q_{(0,1)} :=  L_{\rmd_{A}}=[\rmd, i_{\rmd_{A}}]$ and $Q_{(1,0)} := \rmd$, where $\rmd$ is the canonical de Rham differential on the antitangent bundle.  In homogeneous local coorinates $(x^{a}, \zx^{\alpha}, \rmd x^{b},  \rmd \zx^{\beta})$ we have
\begin{eqnarray*}
L_{\rmd_{A}} &=& \zx^{\alpha}Q_{\alpha}^{a}(x)\frac{\partial}{\partial x^{a}} + \frac{1}{2!} \zx^{\alpha} \zx^{\beta}Q_{\beta \alpha}^{\gamma}\frac{\partial}{\partial \zx^{\gamma}}\\
 &+& \left( (-1)^{\widetilde{\alpha}} \zx^{\alpha} \rmd x^{b}\frac{\partial Q_{\alpha}^{a}}{\partial x^{b}}  - \rmd \zx^{\alpha}Q_{\alpha}^{a}\right)\frac{\partial}{\partial \rmd x^{a}} {-}\left(\rmd \zx^{\alpha}\zx^{\beta}Q_{\beta\alpha}^{\gamma} +  (-1)^{\widetilde{\alpha} + \widetilde{\beta}}\frac{1}{2!}\zx^{\alpha} \zx^{\beta} \rmd x^{b}\frac{\partial Q_{ \beta\alpha}^{\gamma}}{\partial x^{b}}  \right)\frac{\partial}{\partial \rmd \zx^{\gamma}},\\
 \rmd  &=& \rmd x^{a}\frac{\partial}{\partial x^{a}} + \rmd \zx^{\alpha}\frac{\partial}{\partial \zx^{\alpha}}.
\end{eqnarray*}
It is then a matter of direct calculation to see that $\phi_{Q}(x, \zx, \rmd x) =0$. Thus we have  the following.
\begin{theorem}
Any double Lie algebroid of the form $(\Pi \sT \Pi A , L_{\rmd_{A}}, \rmd)$, where $(\Pi A , \rmd_{A})$ is a Lie algebroid, is unimodular.
\end{theorem}

\section*{Acknowledgements}
The author cordially thanks Florian Sch\"{a}tz for his comments on an earlier draft of this work. A special thank you goes to Janusz Grabowski for giving the author  the opportunity to present some of the contents of this paper at the Geometric Methods in Physics seminar in Warsaw on April 26\textsuperscript{th}  2017.

\end{document}